\documentclass[oneside,english]{amsart}
\usepackage[T1]{fontenc}
\usepackage[utf8]{inputenc}
\usepackage{geometry}
\usepackage{amstext}
\usepackage{amsthm}
\usepackage{amssymb}

\makeatletter
\numberwithin{equation}{section}
\numberwithin{figure}{section}
\theoremstyle{plain}
\newtheorem{thm}{\protect\theoremname}
\theoremstyle{plain}
\newtheorem{lem}[thm]{\protect\lemmaname}
\theoremstyle{definition}
\newtheorem{example}[thm]{\protect\examplename}
\theoremstyle{plain}
\newtheorem{cor}[thm]{\protect\corollaryname}

\usepackage{shuffle}

\newcommand{\harub}{\mathbin{\underline{\ast}}}
\newcommand{\shaub}{\mathbin{\underline{\shuffle}}}

\usepackage{babel}
\providecommand{\corollaryname}{Corollary}
  \providecommand{\examplename}{Example}
  \providecommand{\lemmaname}{Lemma}
\providecommand{\theoremname}{Theorem}

\usepackage{babel}
\providecommand{\corollaryname}{Corollary}
  \providecommand{\examplename}{Example}
  \providecommand{\lemmaname}{Lemma}
\providecommand{\theoremname}{Theorem}

\makeatother

\usepackage{babel}
\providecommand{\corollaryname}{Corollary}
\providecommand{\examplename}{Example}
\providecommand{\lemmaname}{Lemma}
\providecommand{\theoremname}{Theorem}

\begin{document}

\title{A cyclic analogue of multiple zeta values}

\author{Minoru Hirose}
\address[Minoru Hirose]{Faculty of Mathematics, Kyushu University 744, Motooka, Nishi-ku,
Fukuoka, 819-0395, Japan}
\email{m-hirose@math.kyushu-u.ac.jp}

\author{Hideki Murahara}
\address[Hideki Murahara]{Nakamura Gakuen University Graduate School, 5-7-1, Befu, Jonan-ku,
Fukuoka, 814-0198, Japan}
\email{hmurahara@nakamura-u.ac.jp}

\author{Takuya Murakami}
\address[Takuya Murakami]{Graduate School of Mathematics, Kyushu University, 744, Motooka,
Nishi-ku, Fukuoka, 819-0395, Japan}
\email{tak\_mrkm@icloud.com}


\begin{abstract}
We consider a cyclic analogue of multiple zeta values (CMZVs), which has two kinds of expressions; series and integral expression.
We prove an `integral$=$series' type identity for CMZVs. 
By using this identity, we construct two classes of $\mathbb{Q}$-linear relations among CMZVs.
One of them is a generalization of the cyclic sum formula for multiple zeta-star values.
We also give an alternative proof of the derivation relation for multiple zeta values. 
\end{abstract}

\maketitle

\section{\label{sec:Introduction}Introduction}

For $k_{1},\dots,k_{r}\in\mathbb{Z}_{\geq1}$ with $k_{r}\geq2$,
the multiple zeta values (MZVs) and the multiple zeta-star values
(MZSVs) are defined by 
\[
\zeta(k_{1},\dots,k_{r}):=\sum_{0<n_{1}<\cdots<n_{r}}\frac{1}{n_{1}^{k_{1}}\cdots n_{r}^{k_{r}}}
\]
and 
\[
\zeta^{\star}(k_{1},\dots,k_{r}):=\sum_{0<n_{1}\leq\cdots\leq n_{r}}\frac{1}{n_{1}^{k_{1}}\cdots n_{r}^{k_{r}}}.
\]
We say that an index $(k_{1},\dots,k_{r})\in\mathbb{Z}_{\geq1}^{r}$
is admissible if $k_{r}\geq2$.

In \cite{NPY_schur}, Nakasuji-Phuksuwan-Yamasaki gave an integral
expression of ribbon type Schur multiple zeta values, which is a generalization
of the `integral$=$series' identity established by Kaneko-Yamamoto \cite{Kaneko_Yamamoto}.
The first main result of this paper is a cyclic
analogue of their results. Let $s\in\mathbb{Z}_{\geq1}$, $r_{1},\dots,r_{s}\in\mathbb{Z}_{\geq1}$,
and $k_{1,1},\dots,k_{1,r_{1}},\dots,k_{s,1},\dots,k_{s,r_{s}}\in\mathbb{Z}_{\geq1}$.
A multi-index $[(k_{1,1},\dots,k_{1,r_{1}}),\dots,(k_{s,1},\dots,k_{s,r_{s}})]$
is called an admissible cyclic index if 
\begin{itemize}
\item for all $1\leq i\leq s$, the index $(k_{i,1},\dots,k_{i,r_{i}})$
is admissible or equal to $(1)$, 
\item there exists $1\leq i\leq s$ such that $(k_{i,1},\dots,k_{i,r_{i}})\neq(1)$. 
\end{itemize}
For an admissible cyclic index $\Bbbk=[(k_{1,1},\dots,k_{1,r_{1}}),\dots,(k_{s,1},\dots,k_{s,r_{s}})]$,
we define the cyclic multiple zeta value (CMZV) by
\begin{equation}
 \zeta_\mathrm{cyc}(\Bbbk):=\sum_{(n_{1,1},\dots,n_{s,r_{s}})\in S}\prod_{i=1}^{s}\prod_{j=1}^{r_{i}}\frac{1}{n_{i,j}^{k_{i,j}}}\label{eq:series_exp},
\end{equation}
where 
\[
S:=\{(n_{1,1},\dots,n_{s,r_{s}})\in\mathbb{Z}_{\geq1}^{r_{1}+\cdots r_{s}}\mid n_{1,1}<\cdots<n_{1,r_{1}}\geq n_{2,1}<\cdots<n_{2,r_{2}}\geq\cdots\geq n_{s,1}<\cdots<n_{s,r_{s}}\geq n_{1,1}\}.
\]

\begin{thm}[Cyclic integral-series identity]
\label{thm:Cyclic_Integral_Series_identity}Let $\Bbbk=[(k_{1,1},\dots,k_{1,r_{1}}),\dots,(k_{s,1},\dots,k_{s,r_{s}})]$
be an admissible cyclic index. Put $k_{i}=\sum_{j=1}^{r_{i}}k_{i,j}$.
Then we have 
\begin{equation}
\zeta_\mathrm{cyc}(\Bbbk)=\int_{D}\prod_{i=1}^{s}\prod_{j=1}^{k_{i}}a_{i,j}dt_{i,j}\label{eq:integral_exp},
\end{equation}
where 
\[
a_{i,j}:=\begin{cases}
\frac{1}{1-t_{i,j}} & j\in\{1,k_{i,1}+1,\dots,k_{i,1}+\cdots+k_{i,r_{i}-1}+1\},\\
\frac{1}{t_{i,j}} & \text{otherwise}
\end{cases}
\]
and 
\[
D:=\{(t_{1,1},\dots,t_{s,k_{s}})\in(0,1)^{k_{1}+\cdots+k_{s}}\mid t_{1,1}<\cdots<t_{1,k_{1}}>t_{2,1}<\cdots<t_{2,k_{2}}>\cdots>t_{s,1}<\cdots<t_{s,k_{s}}>t_{1,1}\}.
\]
\end{thm}

We call (\ref{eq:series_exp}) (resp. (\ref{eq:integral_exp})) as
series (resp. integral) expression of $\zeta_\mathrm{cyc}(\Bbbk)$.

The second and third main theorems (Theorems \ref{thm:CYC-1} and \ref{thm:CYC-2}) are two classes of $\mathbb{Q}$-linear relations among CMZVs. Theorem \ref{thm:CYC-1} is a generalization of the cyclic sum formula for MZSVs which was proved by Ohno-Wakabayashi \cite{Ohno_Wakabayashi}.
We recall Hoffman's algebraic setup with a slightly different convention (see \cite{hoffman_97}).
We put $\mathfrak{h}:=\mathbb{Q}\langle x,y\rangle$.
We define subspaces $\mathfrak{h}_{C},\mathfrak{h}^{0},\mathfrak{h}^{1},\mathfrak{h}_{C}^{0}$
and $\mathfrak{h}_{C}^{1}$ by $\mathfrak{h}_{C}:=\mathfrak{h}x\oplus\mathfrak{h}y$,
$\mathfrak{h}^{0}:=\mathbb{Q}\oplus y\mathfrak{h}x$, $\mathfrak{h}^{1}:=\mathbb{Q}\oplus y\mathfrak{h}$,
$\mathfrak{h}_{C}^{0}:=\mathfrak{h}^{0}\cap\mathfrak{h}_{C}$ and
$\mathfrak{h}_{C}^{1}:=\mathfrak{h}^{1}\cap\mathfrak{h}_{C}$. 
Put $z_{k}:=yx^{k-1}$ for $k\in\mathbb{Z}_{\geq1}$. 
We denote by $\mathfrak{h}^\mathrm{cyc}$ the subspace of $\oplus_{s=1}^{\infty}\mathfrak{h}^{\otimes s}$ 
spanned by 
\[
\bigcup_{s=1}^{\infty}\{u_{1}\otimes\cdots\otimes u_{s}\in\mathfrak{h}^{\otimes s}\mid u_{1},\dots,u_{s}\in\mathfrak{h}_{C}^{0}\cup\{y\}\ \text{and there exists }j\ \text{such that }u_{j}\neq y\}.
\]
We define a $\mathbb{Q}$-linear map $Z_\mathrm{cyc}:\mathfrak{h}^\mathrm{cyc}\to\mathbb{R}$
by 
\[
 Z_\mathrm{cyc}(z_{k_{1,1}}\cdots z_{k_{1,r_{1}}}\otimes\cdots\otimes z_{k_{s,1}}\cdots z_{k_{s,r_{s}}})
 =\zeta_\mathrm{cyc}([(k_{1,1},\dots,k_{1,r_{1}}),\dots,(k_{s,1},\dots,k_{s,r_{s}})]).
\]

\begin{thm}
\label{thm:CYC-1}For $u_{1}\otimes\cdots\otimes u_{s}\in\mathfrak{h}^\mathrm{cyc}$,
we have 
\begin{multline*}
 \sum_{i=1}^{s}Z_\mathrm{cyc}(u_{1}\otimes\cdots\otimes u_{i-1}\otimes(y\shaub u_{i})\otimes u_{i+1}\otimes\cdots\otimes u_{s})
 =\sum_{i=1}^{s}Z_{\mathrm{cyc}}(u_{1}\otimes\cdots\otimes u_{i}\otimes y\otimes u_{i+1}\otimes\cdots\otimes u_{s}),
\end{multline*}
where $y\shaub u_{i}=y\shuffle u_{i}-yu_{i}-u_{i}y$ (see Section \ref{subsec:inner-shuffle-product} for the general definition of $\shaub$). 
\end{thm}

\begin{thm} \label{thm:CYC-2}
 For $u_{1}\otimes\cdots\otimes u_{s}\in\mathfrak{h}^\mathrm{cyc}$
 and $k\in\mathbb{Z}_{\geq1}$, we have 
 \[
  \sum_{i=1}^{s}Z_{\mathrm{cyc}}(u_{1}\otimes\cdots\otimes u_{i-1}\otimes(z_{k}\harub u_{i})\otimes u_{i+1}\otimes\cdots\otimes u_{s})
  =\sum_{i=1}^{s}Z_{\mathrm{cyc}}(u_{1}\otimes\cdots\otimes u_{i}\otimes z_{k}\otimes u_{i+1}\otimes\cdots\otimes u_{s}),
 \]
 where $z_{k}\harub u_{i}=z_{k}*u_{i}-z_{k}u_{i}-u_{i}z_{k}$ 
 (see Section \ref{subsec:inner-harmonic-product} for the general definition of $\harub$). 
\end{thm}

This paper is organized as follows. 
In Sections \ref{sec:proof-of-cyclic-integral-series-identity},
\ref{sec:proof-of-CYC1} and \ref{sec:Proof-of-CYC2}, 
we give proofs of Theorems \ref{thm:Cyclic_Integral_Series_identity}, \ref{thm:CYC-1}
and \ref{thm:CYC-2}, respectively. 
In Section \ref{sec:applications-of-Theorems}, 
we give an alternative proof of the cyclic sum formula for MZSVs (see \cite{Ohno_Wakabayashi}), 
the derivation relation for MZVs (see \cite{ihara_kaneko_zagier_2006}) 
and the sum formula for MZVs as applications of Theorems \ref{thm:Cyclic_Integral_Series_identity} and \ref{thm:CYC-1}.

\section{\label{sec:proof-of-cyclic-integral-series-identity}Proof of cyclic
integral-series identity}
\subsection{Nakasuji-Phuksuwan-Yamasaki's integral-series identity for ribbon
type Schur MZVs}
For the proof of the cyclic integral-series identity, let us introduce the notion of ribbon type Schur MZVs. 
Let 
\[
 \Bbbk=[(k_{1,1},\dots,k_{1,r_{1}}),\dots,(k_{s,1},\dots,k_{s,r_{s}})]
\]
be an admissible cyclic index with $(k_{1,1},\dots,k_{1,r_{1}})\neq (1)$. 
Put $k_{i}=\sum_{j=1}^{r_{i}}k_{i,j}$.
Then the ribbon type Schur MZV $\zeta_{\mathrm{ribbon}}(\Bbbk)$ is defined by 
\[
 \sum_{(n_{1,1},\dots,n_{s,r_{s}})\in S'}\prod_{i=1}^{s}\prod_{j=1}^{r_{i}}\frac{1}{n_{i,j}^{k_{i,j}}},
\]
where 
\[
 S':=\{(n_{1,1},\dots,n_{s,r_{s}})\in\mathbb{Z}_{\geq1}^{r_{1}+\cdots r_{s}}\mid n_{1,1}<\cdots<n_{1,r_{1}}\geq n_{2,1}<\cdots<n_{2,r_{2}}\geq\cdots\geq n_{s,1}<\cdots<n_{s,r_{s}}\}.
\]
In \cite[Section 6.1]{NPY_schur}, Nakasuji-Phuksuwan-Yamasaki gave
a following integral expression: 
\begin{equation}
 \zeta_{\mathrm{ribbon}}(\Bbbk)=\int_{D'}\prod_{i=1}^{s}\prod_{j=1}^{k_{i}}a_{i,j}dt_{i,j}\label{eq:integral_exp_ribbon},
\end{equation}
where 
\[
 a_{i,j}:=
 \begin{cases}
  \frac{1}{1-t_{i,j}} & j\in\{1,k_{i,1}+1,\dots,k_{i,1}+\cdots+k_{i,r_{i}-1}+1\},\\
  \frac{1}{t_{i,j}} & \text{otherwise}
 \end{cases}
\]
and 
\[
 D':=\{(t_{1,1},\dots,t_{s,k_{s}})\in(0,1)^{k_{1}+\cdots+k_{s}}\mid t_{1,1}<\cdots<t_{1,k_{1}}>t_{2,1}<\cdots<t_{2,k_{2}}>\cdots>t_{s,1}<\cdots<t_{s,k_{s}}\}.
\]
Note that $S=S'\cap\{n_{s,r_{s}}\geq n_{1,1}\}$ and $D=D'\cap\{t_{s,k_{s}}>t_{1,1}\}$.

\subsection{Proof of cyclic integral-series identity}
In this section, we prove Theorem \ref{thm:Cyclic_Integral_Series_identity}.
Let 
\[
 \Bbbk=[(k_{1,1},\dots,k_{1,r_{1}}),\dots,(k_{s,1},\dots,k_{s,r_{s}})]
\]
be an admissible cyclic index. Put $\Bbbk_{i}=(k_{i,1},\dots,k_{i,r_{i}})$ and $k_{i}=\sum_{j=1}^{r_{i}}k_{i,j}$.
We denote by $\zeta_\mathrm{cycint}(\Bbbk)$ the integral
expression appeared in Theorem \ref{thm:Cyclic_Integral_Series_identity}.
We prove $\zeta_\mathrm{cyc}(\Bbbk)=\zeta_\mathrm{cycint}(\Bbbk)$
by induction on $s$. Without loss of generality we can assume $\Bbbk_{1}\neq (1)$.
The case $s=1$ is just a usual integral expression
of a multiple zeta value. Note that we have 
\begin{align*}
 & \{(t_{1,1},\dots,t_{s,k_{s}})\in(0,1)^{k_{1}+\cdots+k_{s}}\mid t_{1,1}<\cdots<t_{1,k_1}>t_{2,1}<\cdots<t_{2,k_{2}}>\cdots>t_{s,1}<\cdots<t_{s,k_{s}}\}\\
= & \{(t_{1,1},\dots,t_{s,k_{s}})\in(0,1)^{k_{1}+\cdots+k_{s}}\mid t_{1,1}<\cdots<t_{1,k_1}>t_{2,1}<\cdots<t_{2,k_{2}}>\cdots>t_{s,1}<\cdots<t_{s,k_{s}}\geq t_{1,1}\}\\
 & \sqcup\{(t_{1,1},\dots,t_{s,k_{s}})\in(0,1)^{k_{1}+\cdots+k_{s}}\mid t_{1,1}<\cdots<t_{1,k_1}>t_{2,1}<\cdots<t_{2,k_{2}}>\cdots>t_{s,1}<\cdots<t_{s,k_{s}}<t_{1,1}\}.
\end{align*}
Thus from (\ref{eq:integral_exp_ribbon}), 
\begin{equation}
\zeta_\mathrm{ribbon}([\Bbbk_{1},\dots,\Bbbk_{s}])=\zeta_\mathrm{cycint}([\Bbbk_{1},\dots,\Bbbk_{s}])+\zeta_\mathrm{cycint}([\Bbbk_{s}\Bbbk_{1},\Bbbk_{2},\dots,\Bbbk_{s-1}]).\label{eq:eq1_in_proof_of_cyc_int_ser_identity}
\end{equation}
Here we denote by $\Bbbk_{s}\Bbbk_{1}$ the concatenation of $\Bbbk_{s}$ and $\Bbbk_{1}$, i.e. $\Bbbk_{s}\Bbbk_{1}:=(k_{s,1},\dots,k_{s,r_{s}},k_{1,1},\dots,k_{1,r_{1}})$.
From 
\begin{align*}
 & \{(n_{1,1},\dots,n_{s,r_{s}})\in\mathbb{Z}_{\geq1}^{r_{1}+\cdots r_{s}}\mid n_{1,1}<\cdots<n_{1,r_{1}}\geq n_{2,1}<\cdots<n_{2,r_{2}}\geq\cdots\geq n_{s,1}<\cdots<n_{s,r_{s}}\}\\
= & \{(n_{1,1},\dots,n_{s,r_{s}})\in\mathbb{Z}_{\geq1}^{r_{1}+\cdots r_{s}}\mid n_{1,1}<\cdots<n_{1,r_{1}}\geq n_{2,1}<\cdots<n_{2,r_{2}}\geq\cdots\geq n_{s,1}<\cdots<n_{s,r_{s}}\geq n_{1,1}\}\\
 & \sqcup\{(n_{1,1},\dots,n_{s,r_{s}})\in\mathbb{Z}_{\geq1}^{r_{1}+\cdots r_{s}}\mid n_{1,1}<\cdots<n_{1,r_{1}}\geq n_{2,1}<\cdots<n_{2,r_{2}}\geq\cdots\geq n_{s,1}<\cdots<n_{s,r_{s}}<n_{1,1}\},
\end{align*}
we have
\begin{equation}
 \zeta_\mathrm{ribbon}([\Bbbk_{1},\dots,\Bbbk_{s}])=\zeta_\mathrm{cyc}([\Bbbk_{1},\dots,\Bbbk_{s}])
 +\zeta_\mathrm{cyc}([\Bbbk_{s}\Bbbk_{1},\Bbbk_{2},\dots,\Bbbk_{s-1}]). \label{eq:eq2_in_proof_of_cyc_int_ser_identity}
\end{equation}
From the induction hypothesis, we have 
\begin{equation}
 \zeta_\mathrm{cycint}([\Bbbk_{s}\Bbbk_{1},\Bbbk_{2},\dots,\Bbbk_{s-1}]) 
 =\zeta_\mathrm{cyc}([\Bbbk_{s}\Bbbk_{1},\Bbbk_{2},\dots,\Bbbk_{s-1}]).\label{eq:eq3_in_proof_of_cyc_int_ser_identity}
\end{equation}
From (\ref{eq:eq1_in_proof_of_cyc_int_ser_identity}), (\ref{eq:eq2_in_proof_of_cyc_int_ser_identity}),
(\ref{eq:eq3_in_proof_of_cyc_int_ser_identity}), we have 
\begin{equation}
\zeta_{{\rm cycint}}([\Bbbk_{1},\dots,\Bbbk_{s}])=\zeta_{{\rm cyc}}([\Bbbk_{1},\dots,\Bbbk_{s}]).
\end{equation}
Thus Theorem \ref{thm:Cyclic_Integral_Series_identity} is proved.

\section{\label{sec:proof-of-CYC1}Proof of Theorem \ref{thm:CYC-1}}

\subsection{\label{subsec:inner-shuffle-product}Inner shuffle product}

We define the shuffle product $\shuffle:\mathfrak{h}\times\mathfrak{h}\to\mathfrak{h}$
by 
\[
1\shuffle w=w\shuffle1=w,
\]
\[
uw\shuffle u'w'=u(w\shuffle u'w')+u'(uw\shuffle w'),
\]
where $u,u'\in\{x,y\}$ and $w,w'\in\mathfrak{h}$. 
We define the inner shuffle product $\shaub:\mathfrak{h}\times\mathfrak{h}_{C}\to\mathfrak{h}_{C}$
by 
\[
w\shaub x=w\shaub y=0,
\]
\[
w\shaub uw'u'=u(w\shuffle w')u',
\]
where $u,u'\in\{x,y\}$ and $w,w'\in\mathfrak{h}$. Note that we have
$y\shaub w=y\shuffle w-yw-wy$ for $w\in\mathfrak{h}_{C}$. The following
lemma is a key property of $\shaub$. 
\begin{lem}
\label{lem:inner_shuffle_property}For $0<p<q<1$ , let $f_{p,q}:\mathfrak{h}x\oplus\mathfrak{h}y\to\mathbb{R}$
be a $\mathbb{Q}$-linear map defined by $f_{p,q}(x)=f_{p,q}(y)=0$
and 
\[
f_{t_{1},t_{k}}(u_{1}\cdots u_{k})=\int_{t_{1}<\cdots<t_{k}}a_{1}\cdots a_{k}dt_{2}\cdots dt_{k-1}
\]
for $k>1$, where $u_{1},\dots,u_{k}\in\{x,y\}$ and 
\[
a_{i}=\begin{cases}
\frac{1}{t_{i}} & u_{i}=x,\\
\frac{1}{1-t_{i}} & u_{i}=y.
\end{cases}
\]
Then we have 
\[
f_{p,q}(u_{1}\cdots u_{k}\shaub w)=f_{p,q}(w)\int_{p<t_{1}<\cdots<t_{k}<q}\prod_{i=1}^{k}a_{i}dt_{i},
\]
where $u_{1},\dots,u_{k}\in\{x,y\}$ and 
\[
a_{i}=\begin{cases}
\frac{1}{t_{i}} & u_{i}=x,\\
\frac{1}{1-t_{i}} & u_{i}=y.
\end{cases}
\]
\end{lem}

\begin{example}
When $k=2,u_{1}=x,u_{2}=x$ and $w=yx$, we have

\begin{align*}
f_{p,q}(x^{2}\shaub yx) & =\frac{1}{1-p}\cdot\frac{1}{q}\cdot\int_{p<t_{2}<t_{3}<q}\frac{dt_{2}}{t_{2}}\cdot\frac{dt_{3}}{t_{3}}\\
 & =f_{p,q}(yx)\cdot\int_{p<t_{2}<t_{3}<q}\frac{dt_{2}}{t_{2}}\cdot\frac{dt_{3}}{t_{3}}.
\end{align*}
\end{example}

\subsection{Proof of Theorem \ref{thm:CYC-1}}

For $s\leq s'$, we put 
\[
E(s,s',t)=\begin{cases}
1 & s\leq t\leq s',\\
0 & \text{otherwise}.
\end{cases}
\]
By Lemma \ref{lem:inner_shuffle_property}, we have 
\[
Z_{\mathrm{cyc}}(u_{1}\otimes\cdots\otimes u_{i-1}\otimes(y\shaub u_{i})\otimes u_{i+1}\otimes\cdots\otimes u_{s})=\int_{D}\left(\prod_{c=1}^{s}\prod_{j=1}^{k_{c}}a_{c,j}dt_{c,j}\right)\int_{0<t<1}E(t_{i,1},t_{i,k_{i}},t)\frac{dt}{1-t}.
\]
and 
\[
Z_{\mathrm{cyc}}(u_{1}\otimes\cdots\otimes u_{i}\otimes y\otimes u_{i+1}\otimes\cdots\otimes u_{s})=\int_{D}\left(\prod_{c=1}^{s}\prod_{j=1}^{k_{c}}a_{c,j}dt_{c,j}\right)\int_{0<t<1}E(t_{(i+1\bmod s),1},t_{i,k_{i}},t)\frac{dt}{1-t},
\]
where $(i+1\bmod s)$ means $i+1$ for $1\leq i<s$ and $1$ for $i=s$.
Thus Theorem \ref{thm:CYC-1} follows from 
\[
\sum_{i=1}^{s}E(t_{i,1},t_{i,k_{i}},t)=\sum_{i=1}^{s}E(t_{(i+1\bmod s),1},t_{i,k_{i}},t)\ \ \ (t\in(0,1)).
\]

\section{\label{sec:Proof-of-CYC2}Proof of Theorem \ref{thm:CYC-2}}

\subsection{\label{subsec:inner-harmonic-product}Inner harmonic product}

We define the harmonic product $*:\mathfrak{h}^{1}\times\mathfrak{h}^{1}\to\mathfrak{h}^{1}$
by 
\[
z_{k_{1}}\cdots z_{k_{r}}*z_{l_{1}}\cdots z_{l_{s}}:=\sum_{d=\max(r,s)}^{r+s}\sum_{\substack{f:\{1,\dots,r\}\to\{1,\dots,d\}\\
g:\{1,\dots,s\}\to\{1,\dots,d\}\\
f,g:\text{strictly increasing}\\
{\mathrm{Im}}f\cup{\mathrm{Im}}g=\{1,\dots,d\}
}
}z_{m_{1}}\cdots z_{m_{d}},
\]
where 
\[
 m_{i}=
 \begin{cases}
  k_{f^{-1}(i)} & i\in{\mathrm{Im}}f\setminus{\mathrm{Im}}g,\\
  l_{g^{-1}(i)} & i\in{\mathrm{Im}}g\setminus{\mathrm{Im}}f,\\
  k_{f^{-1}(i)}+l_{g^{-1}(i)} & i\in{\mathrm{Im}}f\cap{\mathrm{Im}}g.
 \end{cases}
\]
Similarly, we define an inner harmonic product $\harub:\mathfrak{h}^{1}\times\mathfrak{h}_{C}^{1}\to\mathfrak{h}_{C}^{1}$ by 
\[
 z_{k_{1}}\cdots z_{k_{r}}\harub z_{l_{1}}\cdots z_{l_{s}}
 :=\sum_{d=\max(r,s)}^{r+s}\sum_{\substack{f:\{1,\dots,r\}\to\{1,\dots,d\}\\
 g:\{1,\dots,s\}\to\{1,\dots,d\}\\
 f,g:\text{strictly increasing}\\
 {\mathrm{Im}}f\cup{\mathrm{Im}}g=\{1,\dots,d\}\\
 \forall i,\ g(1)\leq f(i)\leq g(s)
 }
 }z_{m_{1}}\cdots z_{m_{d}},
\]
where the definition of $m_{i}$ is same as the one in the previous definition. 
Note that we have $z_{k}\harub w=z_{k}*w-z_{k}w-wz_{k}$
for $w\in\mathfrak{h}_{C}^{1}$
, and $u_1\harub (u_2\harub u_3)=(u_1*u_2)\harub u_3$ for $u_1\in \mathfrak{h}^{1}$ and $u_2, u_3\in \mathfrak{h}_{C}^{1}$.
The following lemma is a key property of $\harub$. 
\begin{lem}
 \label{lem:inner_harmonic_property}
 For $0<p\leq q$, let $f_{p,q}:\mathfrak{h}_{C}^{1}\to\mathbb{R}$ be a $\mathbb{Q}$-linear map defined by 
 \[
  f_{p,q}(z_{k_{1}}\cdots z_{k_{r}})=\sum_{\substack{n_{1}<\cdots<n_{r}\\
  n_{1}=p\\
  n_{r}=q
  }
  }n_{1}^{-k_{1}}\cdots n_{r}^{-k_{r}}.
 \]
 Then we have 
 \[
  f_{p,q}(z_{k_{1}}\cdots z_{k_{r}}\harub w)=f_{p,q}(w)\sum_{p\leq n_{1}<\cdots<n_{r}\leq q}n_{1}^{-k_{1}}\cdots n_{r}^{-k_{r}}.
 \]
\end{lem}

\begin{example}
When $r=1$ and $w=z_{l_{1}}z_{l_{2}}$, we have

\begin{align*}
f_{p,q}(z_{k}\harub z_{l_{1}}z_{l_{2}}) & =f_{p,q}(z_{l_{1}+k}z_{l_{2}})+f_{p,q}(z_{l_{1}}z_{k}z_{l_{2}})+f_{p,q}(z_{l_{1}}z_{l_{2}+k})\\
 & =\sum_{p=n_{1}<n_{2}=q}\frac{1}{n_{1}^{l_{1}+k}n_{2}^{l_{2}}}+\sum_{p=n_{1}<n_{2}<n_{3}=q}\frac{1}{n_{1}^{l_{1}}n_{2}^{k}n_{3}^{l_{2}}}+\sum_{p=n_{1}<n_{2}=q}\frac{1}{n_{1}^{l_{1}}n_{2}^{l_{2}+k}}\\
 & =\sum_{p=n_{1}<n_{2}=q}\frac{1}{n_{1}^{l_{1}}n_{2}^{l_{2}}}\sum_{p\leq n\leq q}\frac{1}{n^{k}}\\
 & =f_{p,q}(z_{l_{1}}z_{l_{2}})\sum_{p\leq n\leq q}\frac{1}{n^{k}}.
\end{align*}
\end{example}

\subsection{Proof of Theorem \ref{thm:CYC-2}}

For $p\leq q$, we put 
\[
E(p,q,n)=\begin{cases}
1 & p\leq n\leq q,\\
0 & \text{otherwise}.
\end{cases}
\]

By Lemma \ref{lem:inner_harmonic_property}, we have 
\[
 Z_\mathrm{cyc}(u_{1}\otimes\cdots\otimes u_{i-1}\otimes(z_{k}\harub u_{i})\otimes u_{i+1}\otimes\cdots\otimes u_{s})
 =\sum_{(n_{1,1},\dots,n_{s,r_{s}})\in S}
 \left(\prod_{i=1}^{s}\prod_{j=1}^{r_{i}}\frac{1}{n_{i,j}^{k_{i,j}}}\right)\sum_{n\in\mathbb{Z}_{\geq 1}}\frac{E(n_{i,1},n_{i,r_{i}},n)}{n^{k}}
\]
and 
\[
 Z_{\mathrm{cyc}}(u_{1}\otimes\cdots\otimes u_{i}\otimes z_{k}\otimes u_{i+1}\otimes\cdots\otimes u_{s})
 =\sum_{(n_{1,1},\dots,n_{s,r_{s}})\in S}
 \left(\prod_{i=1}^{s}\prod_{j=1}^{r_{i}}\frac{1}{n_{i,j}^{k_{i,j}}}\right)\sum_{n\in\mathbb{Z}_{\geq 1}}\frac{E(n_{(i+1\bmod s),1},n_{i,r_{i}},n)}{n^{k}},
\]
where $(i+1\bmod s)$ means $i+1$ for $1\leq i<s$ and $1$ for $i=s$.
Thus Theorem \ref{thm:CYC-2} follows from 
\[
 \sum_{i=1}^{s}E(n_{i,1},n_{i,r_{i}},n)=\sum_{i=1}^{s}E(n_{(i+1\bmod s),1},n_{i,r_{i}},n)\ \ \ (n\in\mathbb{Z}_{\geq1}).
\]

\section{\label{sec:applications-of-Theorems}Applications of Theorems \ref{thm:Cyclic_Integral_Series_identity}
and \ref{thm:CYC-1}}
\subsection{Proof of cyclic sum formula for MZSVs}
In this section, we give an alternative proof of the cyclic sum formula
for MZSVs in \cite{Ohno_Wakabayashi} as an a application of Theorem \ref{thm:CYC-1}. 
\begin{lem}
For $k_{1},\dots,k_{s},l\in\mathbb{Z}_{\geq1}$ such that $k_{s}>1$,
we have
\begin{align*}
 \zeta_\mathrm{cyc}([(k_{1}),\dots,(k_{s})]) & =\zeta(k_{1}+\cdots+k_{s}),\\
 \zeta_\mathrm{cyc}([(l,k_{s}),(k_{s-1}),\dots,(k_{1})]) & =\zeta^{\star}(l,k_{1},\dots,k_{s})-\zeta(l+k_{1}+\cdots+k_{s}).
\end{align*}
\end{lem}

\begin{proof}
This is an immediate consequence of the series expression of $\zeta_\mathrm{cyc}$. 
\end{proof}
Fix $k_{1},\dots,k_{s}\in\mathbb{Z}_{\geq1}$ such that $k_{1}+\cdots+k_{s}>s$.
Put $k=k_{1}+\cdots+k_{s}$. 
By Theorem \ref{thm:CYC-1}, we have
\begin{equation}
\sum_{i=1}^{s} Z_\mathrm{cyc}(z_{k_{s}}\otimes\cdots\otimes z_{k_{i+1}}\otimes(y\shaub z_{k_{i}})\otimes z_{k_{i-1}}\otimes\cdots\otimes z_{k_{1}})
=
\sum_{i=1}^{s} Z_\mathrm{cyc}(z_{k_{s}}\otimes\cdots\otimes z_{k_{i+1}}\otimes y\otimes z_{k_{i}}\otimes\cdots\otimes z_{k_{1}}).\label{eq:from_CYC-1}
\end{equation}

By the previous lemma, we have 
\begin{equation}
 Z_\mathrm{cyc}(z_{k_{s}}\otimes\cdots\otimes z_{k_{i+1}}\otimes y\otimes z_{k_{i}}\otimes\cdots\otimes z_{k_{1}})
 =\zeta(k+1)\label{eq:cyc_zeta_only_single}
\end{equation}
for $1\leq i\leq s$. 
Since 
\[
 y\shaub z_{l}=\sum_{j=1}^{l-1}z_{l-j}z_{j+1}
\]
for $l\in\mathbb{Z}_{\geq1}$, we have 
\begin{equation}
 Z_\mathrm{cyc}(z_{k_{s}}\otimes\cdots\otimes z_{k_{i+1}}\otimes(y\shaub z_{k_{i}})\otimes z_{k_{i-1}}\otimes\cdots\otimes z_{k_{1}})
 =\sum_{j=1}^{k_{i}-1}\zeta^{\star}(k_{i}-j,k_{i+1},\dots,k_{s},k_{1},\dots,k_{i-1},j+1)-(k_{i}-1)\zeta(k+1)\label{eq:cyc_zeta_almost_only_single}
\end{equation}
for $1\leq i\leq s$. 
From (\ref{eq:from_CYC-1}), (\ref{eq:cyc_zeta_only_single}) and (\ref{eq:cyc_zeta_almost_only_single}), we have 
\[
 \sum_{i=1}^{s}\sum_{j=1}^{k_{i}-1}\zeta^{\star}(k_{i}-j,k_{i+1},\dots,k_{s},k_{1},\dots,k_{i-1},j+1)=k\zeta(k+1).
\]
This is the cyclic sum formula for MZSVs.

\subsection{Algebraic preliminary}
For $m\ge1$, we define derivation maps $\partial_{m}$ and $\delta_{m}$
on $\mathfrak{h}$ by 
\[
\delta_{m}(x)=0,\ \delta_{m}(y)=yx^{m-1}(x+y),
\]
\[
\partial_{m}(x)=y(x+y)^{m-1}x,\ \partial_{m}(y)=-y(x+y)^{m-1}x.
\]
Note that 
\begin{equation}
\delta_{m}(w)=z_{m}*w-wz_{m}=z_{m}\harub w+z_{m}w\label{eq:delta_m}
\end{equation}
for $w\in \mathfrak{h}_{C}^{1}$. 
We denote by $[A,B]$ a commutator $AB-BA$. 
\begin{lem}
 \label{lem:two_derivation}For $m\geq1$, we have 
 \[
  \sum_{j=1}^{m-1}[\delta_{j},\partial_{m-j}]=(m-1)(\partial_{m}+\delta_{m}).
 \]
\end{lem}

\begin{proof}
 Put $z=x+y$. we define a derivation $s:\mathfrak{h}\to\mathfrak{h}$ by 
 \[
  s(x)=x^{2},\ s(z)=z^{2}.
 \]
 Then we can easily check that 
 \begin{align*}
  [s,\delta_{m}]&=m\delta_{m+1}, \\
  [s,\partial_{m}]&=m\partial_{m+1}.
 \end{align*}
 We prove the lemma by induction on $m$. We can check the case $m\leq2$
 by direct calculation. Take $m\geq3$ and assume that 
 \[
  \sum_{\substack{p+q=m-1\\
  1\leq p,q\leq m-2
  }
  }[\delta_{p},\partial_{q}]=(m-2)(\partial_{m-1}+\delta_{m-1}).
 \]
 From Jacobi identity, we have 
 \begin{align*}
  0= & \sum_{\substack{p+q=m-1\\
  1\leq p,q\leq m-2
  }
  }\left([s,[\delta_{p},\partial_{q}]]+[\partial_{q},[s,\delta_{p}]]+[\delta_{p},[\partial_{q},s]]\right)\\
  = & (m-2)(m-1)(\partial_{m}+\delta_{m})\\
  & -\sum_{\substack{p+q=m-1\\
  1\leq p,q\leq m-2
  }
  }p[\delta_{p+1},\partial_{q}]-\sum_{\substack{p+q=m-1\\
  1\leq p,q\leq m-2
  }
  }q[\delta_{p},\partial_{q+1}]\\
  = & (m-2)(m-1)(\partial_{m}+\delta_{m})-(m-2)\sum_{\substack{p+q=m\\
  1\leq p,q\leq m-1
  }
  }[\delta_{p},\partial_{q}].
 \end{align*}
 Since $m>2$, we obtain 
 \[
  \sum_{\substack{p+q=m\\
  1\leq p,q\leq m-1
  }
  }[\delta_{p},\partial_{q}]=(m-1)(\partial_{m}+\delta_{m}).
 \]
 Thus the claim is proved. 
\end{proof}

\subsection{Proof of derivation relation}

In this subsection, we give an alternative proof of the derivation
relation in \cite{ihara_kaneko_zagier_2006}. We put $\{1\}^{m}:=y^{m}$
and 
\[
\{1\}_{\star}^{m}:=\begin{cases}
1 & m=0,\\
y(x+y)^{m-1} & m>0.
\end{cases}
\]
\begin{lem}
 For $m\geq1$, we have
 \begin{align}
  m\{1\}_{\star}^{m}
  &=\sum_{i=1}^{m}z_{i}*\{1\}_{\star}^{m-i},\label{eq1}\\
  m\{1\}^{m}
  &=\sum_{i=1}^{m}(-1)^{i-1}z_{i}*\{1\}^{m-i},\label{eq2}\\
  \sum_{i=0}^m(-1)^{i}\{1\}_{\star}^{m-i}*\{1\}^{i}
  &=
  \begin{cases}
   1 & m=0,\\
   0 & m>0.\label{eq3}
  \end{cases}
 \end{align}
\end{lem}
\begin{proof} 
 We prove (\ref{eq1}), first.
 \begin{align}
  &\sum_{i=1}^{m}z_{i}*\{1\}_{\star}^{m-i} \nonumber\\
  &=\sum^{m-1}_{i=1}\sum^{m-i}_{d=1}
  \sum_{\substack{k_1+\cdots+k_d=m-i\\k_i\ge1}}z_i\ast z_{k_1}\dots z_{k_d}+z_m\nonumber\\
  &=\sum^{m-1}_{i=1}\sum^{m-i}_{d=1}
  \left(
  \sum^d_{j=1}
  \sum_{\substack{k_1+\cdots+k_d=m-i\\k_l\ge1}}
  z_{k_1}\dots z_{k_j+i}\dots z_{k_d}
  +\sum^d_{j=0}
  \sum_{\substack{k_1+\cdots+k_d=m-i\\k_l\ge1}}
  z_{k_1}\dots z_{k_j}z_iz_{k_{j+1}}\dots z_{k_d}
  \right)
  +z_m \nonumber\\
  &=\sum^{m-1}_{d=1}\sum^{m-d}_{i=1}
  \sum^d_{j=1}
  \sum_{\substack{k_1+\cdots+k_d=m-i\\k_l\ge1}}
  z_{k_1}\dots z_{k_j+i}\dots z_{k_d} 
  +\sum^{m-1}_{d=1}\sum^{m-d}_{i=1}
  \sum^d_{j=0}
  \sum_{\substack{k_1+\cdots+k_d=m-i\\k_l\ge1}}
  z_{k_1}\dots z_{k_j}z_iz_{k_{j+1}}\dots z_{k_d} +z_m \label{eq:5.4}
 \end{align}
 Here we have
 \begin{align*}
  \text{The first term of (\ref{eq:5.4})}
  &=\sum^{m-1}_{d=1}\sum^d_{j=1}
  \sum^{m-d}_{i=1}
  \sum_{\substack{k_1+\cdots+k_d=m-i\\k_l\ge1}}  
  z_{k_1}\dots z_{k_j+i}\dots z_{k_d}\\
  &=\sum^{m-1}_{d=1}\sum^d_{j=1}
  \sum_{\substack{k_1+\cdots+k_d=m\\k_l\ge1}}
  (k_j-1)z_{k_1}\dots z_{k_d}\\
  &=\sum^{m-1}_{d=1}
  \sum_{\substack{k_1+\cdots+k_d=m\\k_l\ge1}} 
  (m-d)z_{k_1}\dots z_{k_d},\\
  \text{The second term of (\ref{eq:5.4})}
  &=\sum^{m-1}_{d=1}\sum^d_{j=0}\sum^{m-d}_{i=1}
  \sum_{\substack{k_1+\cdots+k_d=m-i\\k_l\ge1}}
  z_{k_1}\dots z_{k_j}z_iz_{k_{j+1}}\dots z_{k_d}\\
  &=\sum^{m-1}_{d=1}
  \sum_{\substack{k_1+\cdots+k_{d+1}=m\\k_l\ge1}}
  (d+1)z_{k_1}\dots z_{k_{d+1}}\\
  &=\sum^{m}_{d=2}
  \sum_{\substack{k_1+\cdots+k_d=m\\k_l\ge1}}
  dz_{k_1}\dots z_{k_d}\\
  &=\sum^{m}_{d=1}
  \sum_{\substack{k_1+\cdots+k_d=m\\k_l\ge1}}
  dz_{k_1}\dots z_{k_d}-z_m.
 \end{align*}
 Then we get (\ref{eq1}).

 We prove (\ref{eq2}), next. 
 We note that
 \begin{align*}
  z_1*z_1^{m-1}&=mz_1^m+\sum_{i=1}^{m-1}z_1^{i-1}z_2z_1^{m-1-i},\\
  z_2*z_1^{m-2}&=\sum_{i=1}^{m-1}z_1^{i-1}z_2z_1^{m-1-i}+\sum_{i=1}^{m-2}z_1^{i-1}z_3z_1^{m-2-i},\\
  &\dots\\
  z_{m-1}*z_1&=\sum_{i=1}^2z_1^{i-1}z_{m-1}z_1^{2-i}+z_m,\\
  z_m*1&=z_m
 \end{align*}
 holds. 
 By taking alternating sum, we have (\ref{eq2}).

 The equation (\ref{eq3}) follows from \cite[Proposition 7.1]{kawashima_2009}. 
\end{proof}

We define a $\mathbb{Q}$-linear map $Z:\mathfrak{h}^{0}\to\mathbb{R}$
by 
\[
Z(z_{k_{1}}\cdots z_{k_{r}}):=\zeta(k_{1},\dots,k_{r}).
\]
By Lemma \ref{lem:inner_harmonic_property} and the series expression of $\zeta_\mathrm{cyc}$, we have 
\[
 Z_\mathrm{cyc}(w\otimes\overbrace{y\otimes\cdots\otimes y}^{m\ {\rm times}})
 =Z(\{1\}_{\star}^{m}\harub w)
\]
for $w\in\mathfrak{h}_{C}^{0}$. 
By Theorem \ref{thm:CYC-1}, we
have 
\[
Z_\mathrm{cyc}(y\shaub w\otimes\overbrace{y\otimes\cdots\otimes y}^{m\ {\rm times}})-(m+1)Z_{\mathrm{cyc}}(w\otimes\overbrace{y\otimes\cdots\otimes y}^{m+1\ {\rm times}})=0.
\]
Thus, we get the following corollary of Theorem \ref{thm:CYC-1}. 
\begin{cor} \label{cor:Fwm}
 For $w\in\mathfrak{h}_{C}^{0}$ and $m\ge0$,
 we have 
 \[
  Z(F(w,m))=0,
 \] 
 where $F(w,m):=\{1\}_{\star}^{m}\harub(y\shaub w)-(m+1)\{1\}_{\star}^{m+1}\harub w$.
\end{cor}

In fact, Corollary \ref{cor:Fwm} is essentially derivation relation.
More precisely, the following theorem holds. 
\begin{thm} \label{lem:der_z_1}
 For $w\in\mathfrak{h}_{C}^{0}$ and $m\geq1$, we have 
 \[
  \sum_{i=1}^{m}(-1)^{i-1}F(y^{i-1}\harub w,m-i)=\partial_{m}(w).
 \]
\end{thm}

We prove this theorem in the rest of this section. 
We prepare some lemmas. 
\begin{lem}
 \label{lem:der_z_sum}For $m\geq1$, we have 
 \[
  \sum_{j=1}^{m-1}\partial_{m-j}(z_{j})=-(m-1)z_{m}.
 \]
\end{lem}

\begin{proof}
 By Lemma \ref{lem:two_derivation}, we have 
 \[
  \sum_{j=1}^{m-1}[\delta_{j},\partial_{m-j}](x+y)=(m-1)(\partial_{m}+\delta_{m})(x+y).
 \]
 Since 
 \begin{align*}
  \sum_{j=1}^{m-1}[\delta_{j},\partial_{m-j}](x+y)= & -\sum_{j=1}^{m-1}\partial_{m-j}(yx^{j-1}(x+y))\\
    = & -\sum_{j=1}^{m-1}\partial_{m-j}(z_{j})(x+y)
 \end{align*}
 and 
 \[
  (\partial_{m}+\delta_{m})(x+y)=z_{m}(x+y),
 \]
 we have 
 \[
  \sum_{j=1}^{m-1}\partial_{m-j}(z_{j})=-(m-1)z_{m}. \qedhere
 \]
\end{proof}
\begin{lem} \label{lem:weighted_sum}
 For $m\geq1$, we have 
 \[
  \sum_{i=0}^{m-1}(-1)^{i}(m-i)\{1\}_{\star}^{m-i}*\{1\}^{i}=z_{m}.
 \]
\end{lem}

\begin{proof}
 It follows from the following calculation: 
 \begin{align*}
  & \sum_{i=0}^{m}(-1)^{i}(m-i)\{1\}_{\star}^{m-i}*\{1\}^{i}\\
  = & \sum_{i=0}^{m}(-1)^{i}\left(\sum_{j=1}^{m-i}z_{j}*\{1\}{}_{\star}^{m-i-j}\right)*\{1\}^{i}\ \ \ \ \text{(by (\ref{eq1}))}\\
  = & \sum_{j=1}^{m}z_{j}*\sum_{i=0}^{m-j}(-1)^{i}\{1\}{}_{\star}^{m-j-i}*\{1\}^{i}\\
  = & \sum_{j=1}^{m}z_{j}*
  \begin{cases}
   1 & j=m,\\
   0 & j\neq m
  \end{cases}\qquad \qquad \qquad \qquad \ \ \text{(by (\ref{eq3}))}\\
  = & z_{m}. \qedhere
 \end{align*}
\end{proof}
\begin{proof}[Proof of Theorem \ref{lem:der_z_1}]
Put 
\[
G_{m}(w)=\sum_{i=1}^{m}(-1)^{i-1}F(y^{i-1}\harub w,m-i).
\]
By the definition, we have 
\[
G_{m}(w)=G_{m}'(w)+G_{m}''(w),
\]
where 
\[
G_{m}'(w)=\sum_{i=1}^{m}(-1)^{i-1}\{1\}_{\star}^{m-i}\harub(y\shaub(\{1\}^{i-1}\harub w))
\]
and 
\begin{align*}
G_{m}''(w)= & -\sum_{i=1}^{m}(-1)^{i-1}(m-i+1)\{1\}_{\star}^{m-i+1}\harub(\{1\}^{i-1}\harub w)\\
= & -\left(\sum_{i=1}^{m}(-1)^{i-1}(m-i+1)\{1\}_{\star}^{m-i+1}*\{1\}^{i-1}\right)\harub w\\
= & -z_{m}\harub w.
\end{align*}
Here the last equality follows from Lemma \ref{lem:weighted_sum}.
Thus we have 
\[
G_{m}(w)=\sum_{i=1}^{m}(-1)^{i-1}\{1\}_{\star}^{m-i}\harub(y\shaub(\{1\}^{i-1}\harub w))-z_{m}\harub w.
\]
Now we prove $G_{m}(w)=\partial_{m}(w)$ by induction on $m$. 
The case $m=1$ is obvious. 
For $m\geq2$, we assume that $G_{m-j}(w)=\partial_{m-j}(w)$ for all $1\leq j\leq m-1$. 
By Lemma \ref{lem:two_derivation}, we have
\begin{align}
(m-1)(\partial_{m}+\delta_{m})(w)= & \sum_{j=1}^{m-1}(\delta_{j}\partial_{m-j}-\partial_{m-j}\delta_{j})(w)\nonumber \\
= & \sum_{j=1}^{m-1}(\delta_{j}G_{m-j}-G_{m-j}\delta_{j})(w)\nonumber \\
= & \sum_{j=1}^{m-1}\left(z_{j}\harub G_{m-j}(w)+z_{j}G_{m-j}(w)-G_{m-j}(z_{j}\harub w)-G_{m-j}(z_{j}w)\right).\label{eq:e1}
\end{align}
Here the last equality follows from (\ref{eq:delta_m}). 
Now we have 
\begin{align}
&\sum_{j=1}^{m-1}\left(z_{j}\harub G_{m-j}(w)-G_{m-j}(z_{j}\harub w)\right)\nonumber \\
= & \sum_{j=1}^{m-1}\sum_{i=1}^{m-j}(-1)^{i-1}(z_{j}*\{1\}_{\star}^{m-i-j})\harub(y\shaub(\{1\}^{i-1}\harub w))\nonumber \\
 & -\sum_{j=1}^{m-1}\sum_{i=1}^{m-j}(-1)^{i-1}\{1\}_{\star}^{m-i-j}\harub(y\shaub((z_{j}*\{1\}^{i-1})\harub w))\nonumber \\
 & -\sum_{j=1}^{m-1}z_{j}\harub(z_{m-j}\harub w)+\sum_{j=1}^{m-1}z_{m-j}\harub(z_{j}\harub w)\nonumber \\
= & \sum_{i=1}^{m-1}\sum_{j=1}^{m-i}(-1)^{i-1}(z_j*\{1\}_{\star}^{m-i-j})\harub(y\shaub(\{1\}^{i-1}\harub w))\nonumber \\
 & +\sum_{k=2}^m\sum_{j=1}^{k-1}(-1)^{k-1}\{1\}_{\star}^{m-k}\harub(y\shaub(((-1)^{j-1}z_j*\{1\}^{k-j-1})\harub w))\ \ \ \ (k:=i+j)\nonumber \\
= & \sum_{i=1}^{m-1}(-1)^{i-1}(m-i)\{1\}_{\star}^{m-i}\harub(y\shaub(\{1\}^{i-1}\harub w))\nonumber \\
 & +\sum_{k=2}^{m}(-1)^{k-1}(k-1)\{1\}_{\star}^{m-k}\harub(y\shaub(\{1\}^{k-1}\harub w))\qquad\qquad\qquad\ \ (\text{by (\ref{eq1}), (\ref{eq2})})\nonumber \\
= & (m-1)\sum_{i=1}^{m-1}(-1)^{i-1}\{1\}_{\star}^{m-i}\harub(y\shaub(\{1\}^{i-1}\harub w))\nonumber \\
= & (m-1)G_{m}(w)+(m-1)z_{m}\harub w.\label{eq:e2}
\end{align}
We also have 
\begin{align}
 \sum_{j=1}^{m-1}\left(z_{j}G_{m-j}(w)-G_{m-j}(z_{j}w)\right)
 = & \sum_{j=1}^{m-1}\left(z_{j}\partial_{m-j}(w)-\partial_{m-j} (z_{j}w)\right)\nonumber \\
 = & -\sum_{j=1}^{m-1}\partial_{m-j}(z_{j})w\nonumber \\
 = & (m-1)z_{m}w.\label{eq:e3}
\end{align}
Here the last equality follows from Lemma \ref{lem:der_z_sum}. 
From (\ref{eq:delta_m}), (\ref{eq:e1}), (\ref{eq:e2}), and (\ref{eq:e3}),
we have 
\begin{align*}
 (m-1)(\partial_{m}+\delta_{m})(w)= & (m-1)G_{m}(w)+(m-1)z_{m}\harub w+(m-1)z_{m}w\\
 =& (m-1)(G_{m}(w)+\delta_{m}(w)).
\end{align*}
Thus the claim $\partial_{m}(w)=G_{m}(w)$ is proved. 
\end{proof}

\subsection{Proof of sum formula}
Let $k>r>0$. 
Put $\Bbbk=[(k-r+1),\overbrace{(1),\dots,(1)}^{r-1}]$.
From the series expression (\ref{eq:series_exp}), we have
\begin{align*}
\zeta_\mathrm{cyc}(\Bbbk)=\zeta(k).
\end{align*}
On the other hand, from the integral expression  (\ref{eq:integral_exp}), we have
\begin{align*}
\zeta_\mathrm{cyc}(\Bbbk)
 = & \zeta(y^{r-1}\shaub z_{k-r+1})\\
 = & \sum_{\substack{k_{1}+\cdots+k_{r}=k\\
 k_{1},\dots,k_{r-1}\geq1,k_{r}\geq2.
 }
 }\zeta(k_{1},\dots,k_{r}).
\end{align*}
Then we get the sum formula:
\begin{align*}
\sum_{\substack{k_{1}+\cdots+k_{r}=k\\
 k_{1},\dots,k_{r-1}\geq1,k_{r}\geq2.
 }
 }\zeta(k_{1},\dots,k_{r})=\zeta(k).
\end{align*}

\section*{Acknowledgements}
This work was supported by JSPS KAKENHI Grant Numbers JP18J00982, JP18K13392.

\end{document}